\documentclass[reqno]{amsart}
\usepackage{hyperref}
\newtheorem*{theoA}{Theorem A}
\newtheorem*{theoB}{Theorem B}
\newtheorem*{theoC}{Theorem C}
\newtheorem*{theoD}{Theorem D}
\newtheorem*{theoE}{Theorem E}
\newtheorem*{theoF}{Theorem F}
\newtheorem*{theoG}{Theorem G}
\newtheorem*{theoH}{Theorem H}
\newtheorem*{theoI}{Theorem I}
\newtheorem*{theoJ}{Theorem J}
\newtheorem*{theoK}{Theorem K}

\newtheorem{theo}{Theorem}[section]
\newtheorem{lem}{Lemma}[section]
\newtheorem{cor}{Corollary}[section]

\newtheorem{ques}{Question}[section]

\newtheorem{defi}{Definition}[section]
\newtheorem{rem}{Remark}[section]
\newcommand{\ol}{\overline}
\newcommand{\be}{\begin{equation}}
\newcommand{\ee}{\end{equation}}
\newcommand{\beas}{\begin{eqnarray*}}
	\newcommand{\eeas}{\end{eqnarray*}}
\newcommand{\bea}{\begin{eqnarray}}
\newcommand{\eea}{\end{eqnarray}}

\numberwithin{equation}{section}

\begin{document}
\title[\hfilneg \hfil  Uniqueness of entire functions]
{Uniqueness of entire functions concerning differential-difference polynomials sharing small functions}
\author[G. Haldar\hfil
\hfilneg]
{Goutam Haldar}

%\address{A. Banerjee \newline
%Department of Mathematics, University of Kalyani, West Bengal 741235, India.}
%\email{abanerjee\_kal@yahoo.co.in, abanerjeekal@gmail.com}

\address{Goutam Haldar  \newline
Department of Mathematics, Malda College, Rabindra Avenue, Malda, West Bengal 732101, India.}
\email{goutamiit1986@gmail.com, goutamiitm@gmail.com}

%\thanks{Submitted   Aug.    16,  2017.}

\maketitle
%\noindent Uniqueness of difference-differential polynomials of meromorphic functions sharing a small function IM
%\setcounter{page}{288}

\begin{abstract}
   In this paper, we investigate the uniqueness problem of difference polynomials $f^{n}(z)P(f(z))L_c(f)$ and $g^{n}(z)P(g(z))L_c(g)$, where $L_c(f)=f(z+c)+c_0f(z)$, $P(z)$ is a polynomial with constant coefficients of degree $m$ sharing a small function with the notions of weakly weighted sharing and relaxed weighted sharing and obtained the corresponding results, which improve and extend some recent results due to Sahoo and Biswas (Tamkang J. Math., \textbf{49}(2), 85--97 (2018)).
   \end{abstract}

\section{\textbf{Introduction}}
Let $f$ and $g$ be two non-constant meromorphic functions defined in the open complex plane $\mathbb{C}$. If for some $a\in\mathbb{C}\cup\{\infty\}$, the zero of  $f-a$ and $g-a$ have the same locations as well as same multiplicities, we say that $f$ and $g$ share the value $a$ CM (counting multiplicities). If we do not consider the multiplicities, then $f$ and $g$ are said to share the value $a$ IM (ignoring multiplicities). We adopt the standard notations of the Nevanlinna theory of meromorphic functions (see \cite{Hayman & 1964, Laine & 1993, Yi & Yang & 1995}). For a non-constant meromorphic function $f$, we denote by $T(r,f)$ the Nevanlinna characteristic function of $f$ and by $S(r,f )$ any quantity satisfying $S(r,f)=o\{T(r,f)\}$ as $r\rightarrow\infty$ outside of an exceptional set of finite linear measure. We say that $\alpha(z)$ is a small function of $f$, if $\alpha(z)$ is a meromorphic function satisfying $T(r,\alpha(z))=S(r,f)$. We denote by $ E_{k)}(a,f)$ the set of all a-points of $f$ with multiplicities not exceeding $k$, where an a-point is counted according to its multiplicity. Also we denote by $\ol E_{k)}(a,f)$ the set of distinct a-points of $f$ with multiplicities not exceeding $k$. In addition, we need the following definitions. 
\begin{defi}\cite{Lahiri & JMMS2001}Let $p$ be a positive integer and $a\in\mathbb{C}\cup\{\infty\}$.\begin{enumerate}
		\item[(i)] $N(r,a;f\mid \geq p)$ ($\ol N(r,a;f\mid \geq p)$)denotes the counting function (reduced counting function) of those $a$-points of $f$ whose multiplicities are not less than $p$.\item[(ii)]$N(r,a;f\mid \leq p)$ ($\ol N(r,a;f\mid \leq p)$)denotes the counting function (reduced counting function) of those $a$-points of $f$ whose multiplicities are not greater than $p$.
	\end{enumerate}
\end{defi}
\begin{defi}\cite{Lahiri & Complex Var & 2001}
	Let $k$ be a positive integer or infinity. We denote by $N_k(r,a;f)$ the counting function of a-points of $f$, where an a-point of multiplicity $m$ is counted $m$ times if $m\leq k$ and $k$ times if $m>k$. Then
	\beas N_k(r,a;f)=\ol N(r,a;f)+\ol N(r,a;f\mid\geq 2)+\ldots+ \ol N(r,a;f\mid\geq k).\eeas
	Clearly, $N_1(r,a;f)=\ol N(r,a;f)$.
\end{defi}
\begin{defi}
	Let $a\in \mathbb{C}\cup\{\infty\}$. We denote $N_E(r,a;f,g)$ $(N_E(r,a;f,g))$ by the counting function (reduced counting function) of all common zeros of $f-a$ and $g-a$ with the same multiplicities and by $N_0(r,a;f,g)$ $(N_0(r,a;f, g))$ the counting function (reduced counting function) of all common zeros of $f-a$ and $g-a$ ignoring multiplicities. If
	\beas \ol N(r,a;f)+\ol N(r,a;g)-2\ol N_E(r,a;f,g)=S(r,f)+S(r, g),\eeas then we say that $f$ and $g$ share the value a ``CM". If \beas \ol N(r,a;f)+\ol N(r,a;g)-2\ol N_0(r,a;f,g)=S(r,f )+S(r,g),\eeas then we say that $f$ and $g$ share the value a ``IM".
	Definition 4
\end{defi}
\begin{defi}\cite{Lin & Lin & 2006}
	Let $f$ and $g$ share the value a ``IM" and $k$ be a positive integer or infinity. Then $\ol N^{E}_{k)}(r,a;f,g)$ denotes the reduced counting function of those a-points of $f$ whose multiplicities are equal to the corresponding a-points of $g$, and both of their multiplicities are not greater than $k$. $\ol N^0_{
	(k}(r,a;f,g)$ denotes the reduced counting
	function of those a-points of $f$ which are a-points of $g$, and both of their multiplicities
	are not less than $k$.
\end{defi}
\begin{defi}\cite{Lin & Lin & 2006}
	Let $a\in\mathbb{C}\cup\{\infty\}$ and $k$ be a positive integer or infinity. If
	\beas \ol N(r,a;f\mid \leq k)-\ol N^E_{k)}(r,a;f,g)=S(r,f),\eeas
	\beas \ol N(r,a;g\mid\leq k)-\ol N^E_{k)}(r,a;f,g)=S(r,g),\eeas
	\beas \ol N(r, a; f\mid \geq k+1)-\ol N_0^{(k+1}(r,a;f,g)=S(r, f ),\eeas
	\beas \ol N(r, a; g\mid\geq k+1)-\ol N_0^{(k+1}(r,a;f,g)= S(r, g),\eeas or if $k=0$ and
	\beas \ol N(r,a;f)-\ol N_0(r,a;f,g)=S(r,f),\eeas
	\beas \ol N(r,a;g)-\ol N_0(r,a;f,g)=S(r,g),\eeas
	then we say that $f$ and $g$ share the value a weakly with weight $k$ and we write $f$ and	$g$ share $``(a,k)"$.\end{defi}
\begin{defi}\cite{Banerjee & Mukherjee & 2007}
	Let $k$ be a positive integer and for $a\in \mathbb{C}-\{0\}$, $E_{k)}(a;f)=E_{k)}(a;g)$. Let $z_0$ be a zero of $f(z)- a$ of multiplicity $p$ and a zero of $g(z)-a$ of multiplicity $q$. We denote by $\ol N_L(r,a;f)(\ol N_L(r,a;g))$ the reduced counting function of those $a$-points of $f$ and $g$ for which $p>q\geq k+1(q>p\geq k+1)$, by $\ol N_E^{(k+1}(r,a;f)$ the reduced counting function of those a-points of $f$ and $g$ for which $p=q\geq k+1$, by $\ol N_{f\geq k+1}(r,a;f\mid g\neq a)$ the reduced counting functions of
	those a-points of $f$ and $g$ for which $p\geq k + 1$ and $q=0$.
\end{defi}
\begin{defi}\cite{Banerjee & Mukherjee & 2007}
	Let $k$ be a positive integer and for$a\in \mathbb{C}-\{0\}$, let $f$, $g$ share a ``IM". Let $z_0$ be a zero of $f(z)-a$ of multiplicity $p$ and a zero of $g(z)-a$ of multiplicity $q$. We denote by $\ol N_{f\geq k+1}(r,a;f\mid g=m)$ the reduced counting functions of those $a$-points of $f$ and $g$ for which $p\geq k+1$ and $q=m$. We can define $\ol N_L(r,a;f)(\ol N_L(r,a;g))$ and $\ol N_E^{(k+1}(r,a;f)$ in a similar manner as defined in the previous definition.
\end{defi}

\begin{defi}\cite{Lahiri & Banerjee & 2006}
	Let $a, b\in\mathbb{C}\cup \{\infty\}$. We denote by $N(r,a;f\mid g=b)$ the counting function of those $a$-points of $f$, counted according to multiplicity, which	are $b$-points of $g$.
\end{defi}
\begin{defi}\cite{Lahiri & Banerjee & 2006}
	Let $a, b\in\mathbb{C}\cup \{\infty\}$. We denote by $N(r,a;f\mid g\neq b)$ the counting function of those a-points of $f$, counted according to multiplicity, which	are not the b-points of $g$.
\end{defi}

We define shift and difference operators of $f(z)$ by $f(z+c)$ and $\Delta_cf(z)=f(z+c)-f(z)$, respectively. Note that $\Delta_c^nf(z)=\Delta_c^{n-1}(\Delta_{c}f(z))$, where $c$ is a nonzero complex number and $n\geq2$ is a positive integer. For further generalization of $\Delta_cf$, we now define the difference operator of an entire (meromorphic) function $f$ as $L_c(f)=f(z+c)+c_0f(z)$, where $c_0$ is a non-zero complex constant. Clearly, for the particular choice of the constant $c_0=-1$, we get $L_c(f)=\Delta_cf.$\par 

In 2007, Banerjee and Mukherjee \cite{Banerjee & Mukherjee & 2007} introduced a new type of sharing known as relaxed weighted sharing which is weaker than weakly weighted sharing as follows.
\begin{defi}\cite{Banerjee & Mukherjee & 2007}
	We denote by $N(r,a;f \mid= p; g\mid=q)$ the reduced counting function of common a-points of $f$ and $g$ with multiplicities $p$ and $q$, respectively.
\end{defi}
\begin{defi}\cite{Banerjee & Mukherjee & 2007}
	Let $a\in\mathbb{C}\cup\{\infty\}$ and $k$ be a positive integer or infinity. Suppose that $f$ and $g$ share the value a ``IM". If for $p\neq q$, \beas \sum_{p,q\leq k}N(r,a;f\mid=p;g\mid=q)=S(r),\eeas then we say that $f$ and $g$ share the value a with weight $k$ in a relaxed manner and we write $f$ and $g$ share $(a,k)^*$
\end{defi}
Let $P(z)=a_mz^m+a_{m-1}z^{m-1}+\ldots+a_0$ be a nonzero polynomial of degree $m$, where $a_{m}(\neq 0), a_{m-1}, \ldots, a_0 (\neq 0)$ are complex constants and $m$ is a positive integer.\par 
In 1959, Hayman \cite{Hayman & 1959} proved the following result.
\begin{theoA}\cite{Hayman & 1959}
	Let $f$ be a transcendental entire function and let $n(\geq1)$ be an integer. Then $f^nf^{\prime}=1$ has infinitely many solutions.
\end{theoA}
Regarding uniqueness of the above theorem, Yang and Hua \cite{Yang & Hua & 1997}, in 1997 obtained the following result.
\begin{theoB}\cite{Yang & Hua & 1997}
	Let $f$ and $g$ be two non-constant entire functions, $n\geq6$ a positive integer. If $f^nf^{\prime}$ and $g^ng^{\prime}$	share $1$ CM, then either $f(z)=c_1e^{cz}$, $g(z)=c_2e^{-cz}$, where $c_1$, $c_2$ and $c$ are three constants satisfying $(c_1c_2)^{n+1}c^2=-1$ or $f\equiv tg$ for a constant $t$ satisfying $t^{n+1}=1$.
\end{theoB}
In 2002, Fang and Fang \cite{Fang & Fang & 2002} extends Theorem B in the following manner
\begin{theoC}\cite{Fang & Fang & 2002}
	Let $f$ and $g$ be two non-constant entire functions, and let $n(\geq8)$ be an	integer. If $f^n(f-1)f^{\prime}$ and $g^n(g-1)g^{\prime}$ share 1 CM, then $f\equiv g$.
\end{theoC}
	In 2004, Lin and Yi \cite{Lin & Lin & 2006} extended Theorem C by replacing value sharing with fixed point sharing and obtained the following result.
	\begin{theoD}\cite{Lin & Lin & 2006}
		Let $f$ and $g$ be two transcendental entire functions, and let $n(\geq7)$ be an integer. If $f^n(f-1)f^{\prime}$	and $g^n(g-1)g^{\prime}$ share $z$ CM, then $f\equiv g$.
	\end{theoD}
In 2010, Zhang \cite{Zhang & 2010} obtained the following result by replacing $f^{\prime}(z)$ with $f(z+c)$, where $c$ is a non-zero constant.
\begin{theoE}\cite{Zhang & 2010}
	Let $f$ and $g$ be two transcendental entire functions of finite order, and $\alpha(z)(\not\equiv 0,\infty)$ be a small function with respect to both $f$ and $g$. Suppose that $c$ is a nonzero complex constant and $n\geq 7$ be an integer. If $f^n(z)(f(z)-1)f(z+c)$ and $g^n(z)(g(z)-1)g(z+c)$ share $\alpha(z)$ CM, then $f(z)\equiv g(z)$.
\end{theoE}
Recently, Meng \cite{Meng & 2014} improved Theorem E by relaxing the nature of sharing the small function and obtained following results.
\begin{theoF}\cite{Meng & 2014}
	Let $f$ and $g$ be two transcendental entire functions of finite order, and $\alpha(z)(\not\equiv 0, \infty)$ be a small function with respect to both $f$ and $g$. Suppose that $c$ is a nonzero complex constant and $n\geq7$ is an integer. If $f^n(z)(f(z)-1)f(z +c)$ and $g^n(z)(g(z)-1)g(z+c)$ share $``(\alpha(z),2)"$, then $f(z)\equiv g(z)$.
\end{theoF}
\begin{theoG}\cite{Meng & 2014}
	Let $f$ and $g$ be two transcendental entire functions of finite order, and $\alpha(z)(\not\equiv 0, \infty)$  be a small function with respect to both $f$ and $g$. Suppose that $c$ is a nonzero complex constant and $n\geq10$ is an integer. If $f^n(z)(f(z)-1)f (z+c)$ and $g^n(z)(g(z)-1)g(z+c)$ share $(\alpha(z), 2)^*$, then $f(z)\equiv g(z)$.
\end{theoG}
\begin{theoH}\cite{Meng & 2014}
	Let $f(z)$ and $g(z)$ be two transcendental entire functions of finite order, and $\alpha(z)(\not\equiv 0, \infty)$ be a small function with respect to both f $(z)$ and $g(z)$. Suppose that $c$ is a nonzero complex constant and $n\geq16$ is an integer. If $\ol E_{2)}(\alpha(z), f^n(z)(f(z)-1)f(z+c))=\ol E_{2)}(\alpha(z), g^n(z)(g(z)-1)g(z+c))$, then $f(z)\equiv g(z)$.
\end{theoH}
In 2018, Sahoo and Biswas \cite{Sahoo & Biswas & 2018} further extended Theorems E--H in the following.
\begin{theoI}\cite{Sahoo & Biswas & 2018}
	Let $f(z)$ and $g(z)$ be two transcendental entire functions of finite order, and $\alpha(z)(\not\equiv 0, \infty)$ be a small function with respect to both $f(z)$ and $g(z)$ with finitely many zeros. Suppose that $c$ is
	a non-zero complex constant, $n$, $k(\geq 0)$ and $m(\geq 1)$ are integers such that $n\geq 2k+m+6$. If $(f^n(z)(f^m(z)-1)f(z+c))^{(k)}$ and $(g^n(z)(g^m(z)-1)g(z+c))^{(k)}$ share $``(\alpha(z),2)"$, then $f(z)\equiv tg(z)$, where $t^m=1$.
\end{theoI}
\begin{theoJ}\cite{Sahoo & Biswas & 2018}
	Let $f(z)$ and $g(z)$ be two transcendental entire functions of finite order, and $\alpha(z)(\not\equiv 0, \infty)$ be a small function with respect to both $f(z)$ and $g(z)$ with finitely many zeros. Suppose that $c$ is
	a non-zero complex constant, $n$, $k(\geq 0)$ and $m(\geq 1)$ are integers such that $n\geq 3k+2m+8$. If $(f^n(z)(f^m(z)-1)f(z+c))^{(k)}$ and $(g^n(z)(g^m(z)-1)g(z+c))^{(k)}$ share $(\alpha(z),2)^*$, then $f(z)\equiv tg(z)$, where $t^m=1$.
\end{theoJ}
\begin{theoK}\cite{Sahoo & Biswas & 2018}
	Let $f(z)$ and $g(z)$ be two transcendental entire functions of finite order, and $\alpha(z)(\not\equiv 0, \infty)$ be a small function with respect to both $f(z)$ and $g(z)$ with finitely many zeros. Suppose that $c$ is
a non-zero complex constant, $n$, $k(\geq 0)$ and $m(\geq 1)$ are integers such that $n\geq 5k+4m+12$. If $\ol E_{2)}(\alpha(z), (f^n(z)(f^m(z)-1)f(z+c))^{(k)})=\ol E_{2)}(\alpha(z), (g^n(z)(g^m(z)-1)g(z+c))^{(k)})$, then $f(z)\equiv tg(z)$, where $t^m=1$.
\end{theoK}
Since no attempts, till now, have so far been made by any researchers investigating the uniqueness problem by replacing $f(z+c)$ with $L_c(f)$ in all the above-stated Theorems F--K, it is, therefore, inevitable to ask the following question.
\begin{ques}
	What can be said about the uniqueness of $f$ and $g$ if one replace $f^n(f^m-1)f(z+c)$ by the difference polynomial $f(z)^nP(f(z))L_c(f)$ in Theorems F--K?
\end{ques}
%\begin{ques}Can we remove the condition that $\alpha(z)$ has only finitely many zeros in Theorems I--K?\end{ques}
In this paper, we paid our attention to the above question and proved the following three theorems that improve and extend Theorems I--K, respectively. Indeed, the following theorems are the main results of the paper.
\begin{theo}\label{t1}
    Let $f(z)$ and $g(z)$ be two transcendental entire functions of finite order, $\alpha(z)(\not\equiv 0, \infty)$ be a small function with respect to both $f(z)$ and $g(z)$. Suppose $c$ be a non-zero complex constant, $n$, $k(\geq0)$, $m(\geq k+1)$ are integers such that $n\geq 2k+m+6$. If $(f(z)^{n}P(f(z))L_{c}(f))^{(k)}$ and $(g(z)^{n}P(g(z))L_{c}(g))^{(k)}$ share $``(\alpha(z),2)"$, then one of the following two conclusions can be realized.
   \begin{enumerate}
   	\item[\emph{(a)}] $f(z)\equiv tg(z)$, where $t$ is a constant such that $t^d=1$, $d=\text{gcd}(\lambda_0,\lambda_1,\ldots,\lambda_m)$, where $\lambda_j$'s are defined by \beas \lambda_j=\begin{cases} n+1+j, \;\;\;\text{if}\; a_j\neq0\\ n+1+m, \; \text{if}\; a_j=0, \end{cases}j=0,1,\ldots,m.\eeas 
   	\item[\emph{(b)}] $f$ and $g$ satisfy the algebraic equation $R(w_1, w_2)=0$, where $R(w_1, w_2)$ is given by $ R(w_1,w_2)=w_1^{n}P(w_1)L_c(w_1)-w_2^{n}P(w_2)L_c(w_2).$
   \end{enumerate}  
   \end{theo}
\begin{theo}\label{t2}
	 Let $f(z)$ and $g(z)$ be two transcendental entire functions of finite order, $\alpha(z)(\not\equiv 0, \infty)$ be a small function with respect to both $f(z)$ and $g(z)$. Suppose $c$ be a non-zero complex constant, $n$, $k(\geq0)$, $m(\geq k+1)$ are integers such that $n\geq 3k+2m+8$. If $(f(z)^{n}P(f(z))L_{c}(f))^{(k)}$ and $(g(z)^{n}P(g(z))L_{c}(g))^{(k)}$ share $(\alpha(z),2)^*$, then the conclusions of Theorem \ref{t1} holds.
\end{theo}
\begin{theo}\label{t3}
 Let $f(z)$ and $g(z)$ be two transcendental entire functions of finite order, $\alpha(z)(\not\equiv 0, \infty)$ be a small function with respect to both $f(z)$ and $g(z)$. Suppose $c$ be a non-zero complex constant, $n$, $k(\geq0)$, $m(\geq k+1)$ are integers such that $n\geq 9+(7k+5m)/2$. If
$E_{2)}(\alpha(z),(f^nP(f(z))L_c(f))^{(k)})=E_{2)}(\alpha(z),(g^n(z)P(g(z))L_c(g))^{(k)})$, then the conclusions of Theorem \ref{t1} holds.
\end{theo}
\begin{rem}
	Clearly, for the particular choice of $c_0=0$, $L_c(f)$ becomes $f(z+c)$, and therefore Theorem \ref{t1} coincides with Theorem I and Theorem \ref{t2} with Theorem J.
\end{rem}
\begin{rem}
	In Theorem K, if one replace $\ol E_{2)}(\alpha(z),f(z))=\ol E_{2)}(\alpha(z),g(z))$ by $ E_{2)}(\alpha(z),f(z))= E_{2)}(\alpha(z),g(z))$, for any two non-constant meromorphic functions $f$ and $g$, then the lower bound of $n$ can significantly reduced. Regarding this observation, we proved Theorem \ref{t3}. 
\end{rem}
Since for a particular choice of $c_0=-1$, $L_c(f)=\Delta_cf$, we observe the following corollaries. 
\begin{cor}\label{c1}
	Let $f(z)$ and $g(z)$ be two transcendental entire functions of finite order, $\alpha(z)(\not\equiv 0, \infty)$ be a small function with respect to both $f(z)$ and $g(z)$. Suppose $c$ be a non-zero complex constant, $n$, $k(\geq0)$, $m(\geq k+1)$ are integers such that $n\geq 2k+m+6$. If $(f(z)^{n}P(f(z))\Delta_cf)^{(k)}$ and $(g(z)^{n}P(g(z))\Delta_cg)^{(k)}$ share $``(\alpha(z),2)"$, then conclusion of Theorem \ref{t2} holds.
\end{cor}
\begin{cor}\label{c2}
	Let $f(z)$ and $g(z)$ be two transcendental entire functions of finite order, $\alpha(z)(\not\equiv 0, \infty)$ be a small function with respect to both $f(z)$ and $g(z)$. Suppose $c$ be a non-zero complex constant, $n$, $k(\geq0)$, $m(\geq k+1)$ are integers such that $n\geq 3k+2m+8$. If $(f(z)^{n}P(f(z))\Delta_cf)^{(k)}$ and $(g(z)^{n}P(g(z))\Delta_cg)^{(k)}$ share $(\alpha(z),2)^*$, then the conclusions of Theorem \ref{t1} holds.
\end{cor}
\begin{cor}\label{c3}
	Let $f(z)$ and $g(z)$ be two transcendental entire functions of finite order, $\alpha(z)(\not\equiv 0, \infty)$ be a small function with respect to both $f(z)$ and $g(z)$. Suppose $c$ be a non-zero complex constant, $n$, $k(\geq0)$, $m(\geq k+1)$ are integers such that $n\geq 9+(7k+5m)/2$. If
	$E_{2)}(\alpha(z),(f^nP(f(z))\Delta_cf)^{(k)})=E_{2)}(\alpha(z),(g^n(z)P(g(z))\Delta_cg)^{(k)})$, then the conclusions of Theorem \ref{t1} holds.
\end{cor}
\section{\textbf{Some Lemmas}} We now prove several lemmas which will play key roles in proving the main results of the paper. Let $\mathcal{F}$ and $\mathcal{G}$ be two non-constant meromorphic functions. Henceforth we shall denote by $\mathcal{H}$ the following function \be\label{e3.1}H=\left(\frac{\;\;F^{\prime\prime}}{F^{\prime}}-\frac{2F^{\prime}}{F-1}\right)-\left(\frac{\;\;G^{\prime\prime}}{G^{\prime}}-\frac{2G^{\prime}}{G-1}\right).\ee

\begin{lem}\label{lem3.1}\cite{Chiang & Feng & 2008}
Let $f(z)$ be a meromorphic function of finite order $\rho$, and let $c$ be a fixed non-zero complex constant. Then for each $\epsilon>0$, we have \beas T (r,f(z+c))=T(r,f)+O(r^{\rho-1+\epsilon})+O{\log r}.\eeas 
\end{lem}
%\begin{lem}\label{lem3.2}\cite{Luo & Lin & 2011}
	%Let f (z) be a meromorphic function of finite order ρ and let c be a fixed nonzero complex constant. Then \beas N(r,\infty;f(z+c))=N(r,\infty;f(z))+S(r,f),\eeas  \beas N(r,0;f(z+c))=N(r,0;f(z))+S(r,f)\eeas 	\beas \ol N(r,\infty;f(z+c))=\ol N(r,\infty;f(z))+S(r,f),\eeas  \beas \ol N(r,0;f(z+c))=\ol N(r,0;f(z))+S(r,f),\eeas outside of possible exceptional set with finite logarithmic measure\end{lem}
\begin{lem}\label{lem3.3} \cite{Chiang & Feng & 2008}
	Let $f(z)$ be a meromorphic function of finite order $\rho$ and let $c$ be a non-zero complex number. Then for each $\epsilon>0$, we have
	\beas m\left(r,\frac{f(z+c)}{f(z)}\right)+m\left(r, \frac{f(z)}{f(z+c)}\right)=O(r^{\rho-1+\epsilon}).\eeas
\end{lem}
\begin{lem}\label{lem3.3a}\cite{Mohonko & 1971} Let $f$ be a non-constant meromorphic function and let $\mathcal{R}(f)=\sum\limits _{i=0}^{n} a_{i}f^{i}/\sum \limits_{j=0}^{m} b_{j}f^{j}$ be an irreducible rational function in $f$ with constant coefficients $\{a_{i}\}$ and $\{b_{j}\}$ where $a_{n}\not=0$ and $b_{m}\not=0$. Then $$T(r,\mathcal{R}(f))=d\;T(r,f)+S(r,f),$$ where $d=\max\{n,m\}$.\end{lem}
\begin{lem}\label{lem3.3b}\cite{Lahiri & Dewan & 2003} If $N(r,0;f^{(k)}\mid f\not=0)$ denotes the counting function of those zeros of  $f^{(k)}$ which are not the zeros of $f$, where a zero of $f^{(k)}$ is counted according to its multiplicity then $$N\left(r,0;f^{(k)}\mid f\not=0\right)\leq k\ol N(r,\infty;f)+N\left(r,0;f\mid <k\right)+k\ol N\left(r,0f\mid\geq k\right)+S(r,f).$$\end{lem}
\begin{lem}\label{lem3.4}
Let $F=f(z)^n(z)P(f(z))L_c(f)$, where $f(z)$ is an entire function of finite order, and $f(z)$, $f(z+c)$ share $0$ CM. Then \beas T(r,F)=(n+m+1)T(r,f)+S(r,f).\eeas
\end{lem}
\begin{proof}
	Keeping in view of Lemmas \ref{lem3.1} and \ref{lem3.3a}, we have
	\beas T(r,F)&=&T(r,f(z)^nP(f(z))L_c(f))=m(r,f^nP(f)L_c(f))\\&\leq& m(r,f(z)^nP(f(z)))+m(r,L_c(f))+S(r,f)\\&\leq& T(f(z)^nP(f(z)))+m\left(r,\frac{L_c(f)}{f(z)}\right)+m(r,f(z))+S(r,f).\eeas
	i.e., \beas T(r,F)\leq (n+m+1)T(r,f)+S(r,f).\eeas Since $f(z)$ and $f(z+c)$ share $0$ CM, we must have $N\left(r,\infty;\frac{L_c(f)}{f(z)}\right)=S(r,f)$. So, keeping in view of Lemmas \ref{lem3.3} and \ref{lem3.3a}, we have 
	\beas && (n+m+1)T(r,f)=T(r,f(z)^{n+1}P(f(z)))=m(r,f(z)^{n+1}P(f(z)))\\&=&m\left(r,F\frac{f(z)}{L_c(f)}\right)\leq m(r,F)+m\left(r,\frac{f(z)}{L_c(f)}\right)+S(r,f)\\&\leq& T(r,F)+T\left(r, \frac{L_c(f)}{f(z)}\right)+S(r,f)= T(r,F)+N\left(r, \frac{L_c(f)}{f(z)}\right)\\&&+m\left(r, \frac{L_c(f)}{f(z)}\right)+S(r,f)= T(r,F)+S(r,f).\eeas From the above two inequalities, we must have \beas T(r,F)=(n+m+1)T(r,f)+S(r,f).\eeas
	\end{proof}
%\begin{lem}\label{lem3.5}\cite{Lin & Lin & 2006}
	%Let $m$ be a nonnegative integer or $\infty$. Let $F$ and $G$ be two non-constant meromorphic functions, and $F$, $G$ share $``(1,m)"$. If $H\not\equiv 0$, then 
	 %\begin{enumerate}
		%\item[(i).] If $m\geq 2$, then \beas T(r,F)\leq N_2(r,0;F)+N_2(r,\infty;F)+N_2(r,0;G)+N_2(r,\infty;G)+S(r,F)+S(r,G).\eeas
		%\item[(ii).] If $m=1$, then \beas  T(r,F)&\leq& N_2(r,0;F)+N_2(r,\infty;F)+N_2(r,0;G)+N_2(r,\infty;G)\\&&+\ol N_{L}(r,1;F)+S(r,F)+S(r,G).\eeas
		%\item[(iii).] If $m=0$, then \beas  T(r,F)&\leq& N_2(r,0;F)+N_2(r,\infty;F)+N_2(r,0;G)+N_2(r,\infty;G)\\&&+2\ol N_{L}(r,1;F)+\ol N_{L}(r,1;G)+S(r,F)+S(r,G).\eeas
	%\end{enumerate}\end{lem}
%\begin{lem}\label{lem3.6}\cite{Yi & 1995}
	%Let $F$ and $G$ be two non-constant meromorphic functions and $H\equiv 0$. If \beas \limsup_{r\rightarrow \infty}\frac{\ol N(r,0;F)+\ol N(r,\infty;F)+\ol N(r,0;G)+\ol N(r,\infty;G)}{T(r)}<1,\eeas where $T(r)=\text{max}\{T(r,F),T(r,G)\}$, $r\in I$ and $I$ is a set with infinite linear measure, then either $F\equiv G$ or $FG\equiv1$.\end{lem}
\begin{lem}\label{lem3.7}\cite{Yang & 1993}
	Let $f(z)$ and $g(z)$ be two non-constant meromorphic functions. Then \beas N\left(r,\infty;\frac{f}{g}\right)-N\left(r,\infty;\frac{g}{f}\right)=N(r,\infty;f)+N(r,0;g)-N(r,\infty;g)+N(r,0;f).\eeas
\end{lem}
\begin{lem}\label{lem3.8}
	Let $f(z)$ be a transcendental entire function of finite order, $c\in\mathbb{C}$--$\{0\}$ be finite complex constants and $n\in\mathbb{N}$. Let $F(z)=f(z)^nP(f(z))L_c(f)$, where $L_c(f)\not\equiv0$. Then we have \beas (n+m)T(r,f)\leq T(r,F)-N(r,0;L_c(f))+S(r,f).\eeas
\end{lem}
\begin{proof}
	Using Lemmas \ref{lem3.3} and \ref{lem3.7}, we get	\beas && m(f(z)^{n+1}P(f(z)))=m\left(r,\frac{f(z)F}{L_c(f)}\right)\leq m(r,F)+m\left(r,\frac{f(z)}{L_c(f)}\right)+S(r,f)\\&\leq& m(r,F)+T\left(r,\frac{f(z)}{L_c(f)}\right)-N\left(r,\infty;\frac{f(z)}{L_c(f)}\right)+S(r,f)\\&\leq& m(r,F)+T\left(r,\frac{L_c(f)}{f(z)}\right)-N\left(r,\infty;\frac{f(z)}{L_c(f)}\right)+S(r,f)\\&\leq& m(r,F)+N\left(r,\infty;\frac{L_c(f)}{f(z)}\right)+m\left(r,\frac{L_c(f)}{f(z)}\right)-N\left(r,\infty;\frac{f(z)}{L_c(f)}\right)+S(r,f)\\&\leq& m(r,F)+N(r,0; f)-N(r,0;L_c(f))+S(r,f).\eeas
	i.e., \beas m(f(z)^{n+1}P(f(z)))
	\leq T(r,F)+T(r,f)-N(r,0;L_c(f))+S(r,f).\eeas By Lemma \ref{lem3.3a}, we get \beas (n+m+1)T(r,f)=m(r,f^{n+1}P(f))\leq T(r,F)+T(r,f)-N(r,0;L_c(f))+S(r,f).\eeas i.e., \beas (n+m)T(r,f)\leq T(r,F)-N(r,0;L_c(f))+S(r,f).\eeas\end{proof}
\begin{lem}\label{lem3.10}\cite{Banerjee & Mukherjee & 2007}
	Let $F$ and $G$ be two non-constant meromorphic functions that share $(1,2)^*$. Then 
	\beas &&\ol N_L(r,1;F)+\ol N_{F\geq3}(r,1;g|=1)\\&\leq& \ol N(r,0;F)+N(r,\infty;F)-\sum_{p=3}^{\infty}\ol N\left(r,0;\frac{F^{\prime}}{F}\mid\geq p\right)-\ol N_0^2(r,0;F^{\prime})+S(r),\eeas where by $N_0^2(r,0;F^{\prime})$ is the counting function of those zeros of $F^{\prime}$	which
	are not the zeros of $F(F-1)$, where each simple zero is counted once and all other zeros are counted two times.
\end{lem}
\begin{lem}\label{lem3.11}
	Let $F$ and $G$ be two non-constant meromorphic functions such that $E_{2)}(1,F)=E_{2)}(1,G)$ and $H\not\equiv0$. Then 
	\beas && N(r,\infty; H)\\&\leq& \ol N(r, 0;F\mid\geq2)+\ol N(r,0;G\mid\geq2)+\ol N_L(r,1;F)+\ol N_L(r,1;G)+\ol N(r,\infty;F\mid\geq2)\\&&+\ol N(r,\infty;G\mid\geq2)+\ol N_{F\geq3}(r,1;F\mid G\neq1)+\ol N_{G\geq3}(r,1;G\mid F\neq1)\\&&+\ol N_0(r,0;F^{\prime})+\ol N_0(r,0;G^{\prime})+S(r,F)+S(r,G).\eeas
	\end{lem}
\begin{proof}
It can be easily verified that all possible poles of $H$ occur at (i) multiple zeros of $F$ and $G$, (ii) multiple poles of $F$ and $G$, (iii) the common zeros of $F-1$ and $G-1$ whose multiplicities are different, (iii) those $1$-points of $F (G)$ which are not the $1$-points of $F(G)$, (iv) zeros of $F^{\prime}$ which are not the zeros of $F(F-1)$, (v) zeros of $G^{\prime}$ which are not zeros of $G(G-1)$. Since all the poles of $H$ are simple the lemma follows from above. This proves the lemma.
	\end{proof}
\begin{lem}\label{lem3.12}\cite{Banerjee & Mukherjee & 2007}
	If $f$, $g$ be share $``(1,1)"$ and $H\not\equiv 0 $, then
	\beas N(r,1; f\mid\leq 1)\leq N(r,0;H)+S(r,f)\leq N(r,\infty; H)+S(r,f)+S(r,g).\eeas
\end{lem}
\begin{lem}\label{lem3.13}\cite{Banerjee & Mukherjee & 2007}
If $f$, $g$ be two non-constant meromorphic functions such that $E_{1)}(1;f)=E_{1)}(1;g)$ and $H\not\equiv 0$, then
\beas N(r,1;f\mid\leq1)\leq N(r,0;H)\leq N(r,\infty; H)+S(r,f)+S(r,g).\eeas
\end{lem}
\begin{lem}\label{lem3.14}\cite{Banerjee & Mukherjee & 2007}
	If $f$, $g$ be share $(1,1)^*$ and $H\not\equiv0$, then
	\beas N^E(r,1;f,g\mid\leq1)\leq N(r,0; H)\leq N(r,\infty;H)+S(r,f)+S(r, g).\eeas
\end{lem}
\begin{lem}\label{lem3.15}\cite{Banerjee & Mukherjee & 2007}
	If $f$, $g$ be share $(1,1)^*$ and $H\not\equiv0$, then
	\beas N(r,\infty; H)&\leq& \ol N(r,0;f\mid\geq2)+\ol N(r,0;g\mid\geq2)+\ol N(r,\infty;f\mid\geq2)+\ol N_*(r, 1; f,g)\\&&+\ol N(r,\infty;g\mid\geq 2)+\ol N_0(r,0;f^{\prime})+\ol N_0(r,0;g^{\prime})+S(r,f)+S(r,g),\eeas where $\ol N_0(r,0;f^{\prime})$ is the reduced counting function of those zeros of $f^{\prime}$	which are	not the zeros of $f(f-1)$ and $\ol N_0(r,0;g^{\prime})$ is similarly defined.
\end{lem}
\begin{lem}\label{lem3.16}\cite{Banerjee & Mukherjee & 2007}
	Let $E_{2)}(1;f)=E_{2)}(1;g)$. Then
	\beas &&\ol N_{f\geq3}(r,1;f\mid g\neq1)\\&\leq& \frac{1}{2}\ol N(r,0;f)+\frac{1}{2}\ol N(r,\infty;f)-\frac{1}{2}\sum_{p=3}^{\infty}\ol N\left(r,0;\frac{f^{\prime}}{F}\mid\geq p\right)-\frac{1}{2}\ol N_0^2(r,0;f^{\prime})+S(r),\eeas 
\end{lem}
\begin{lem}\label{lem3.17}\cite{Zhang & Yang & 2007}
	Let $f$ be a non-constant meromorphic function, and $p$, $k$ be positive integers.
	Then \beas N_p(r,0;f^{(k)})\leq T(r,f^{(k)}´-T(r,f)+N_{p+k}(r,0;f)+S(r,f),\eeas
	\beas N_p(r,0;f^{(k)})\leq k\ol N(r,\infty;f)+N_{p+k}(r,0;f)+S(r,f).\eeas
\end{lem}

\section{\textbf{Proofs of the theorems}}
\begin{proof}[Proof of Theorem \ref{t1}]
	Let $F=\displaystyle
	\frac{F_1^{(k)}}{\alpha(z)}$ and $G=\displaystyle
	\frac{G_1^{(k)}}{\alpha(z)}$, where $F_1=f^n(z)P(f(z))L_c(f)$ and $G_1=(g^n(z)P(g(z))L_c(g))$.
	Then $F$ and $G$ are two transcendental meromorphic functions that share ``$(1,2)$" except
	the zeros and poles of $\alpha(z)$. We consider the following two cases.\par 
\textbf{Case 1:} Suppose $H\not\equiv 0$. Since $F$ and $G$ share $``(1,2)"$, it follows that $F$ and $G$ share $(1,1)^*$. Keeping in view of Lemmas \ref{lem3.12} and \ref{lem3.15}, we see that \bea\label{e4.1} \ol N(r,1;F)&=&N(r,1;F\mid\leq 1)+\ol N(r,1;F\mid\geq2)\leq N(r,\infty;H)+\ol N(r,1;F\mid\geq2)\nonumber\\&\leq& \ol N(r,0;F\mid\geq2)+\ol N(r,0;G\mid\geq2)+\ol N_*(r, 1; F,G)+\ol N(r,1;F\mid\geq2)\nonumber\\&&\ol N_0(r,0;F^{\prime})+\ol N_0(r,0;G^{\prime})+S(r,F)+S(r,G).\eea 
Since $F$, $G$ share $``(1,2)"$, we must have $\ol N_{F\geq2}(r,1;F\mid G\neq1)=S(r,F)$ and $\ol N(r,1;F\mid\geq2,G\mid=1)=S(r,F)$. Therefore, keeping in view of the above observation and Lemma \ref{lem3.3b}, we get
\bea\label{e4.2} && \ol N_{0}(r,0;G^{\prime})+\ol N(r,1;F\mid\geq2)+\ol N_*(r,1;F,G)\nonumber\\&\leq&\ol N_{0}(r,0;G^{\prime})+\ol N(r,1;F\mid\geq3)+\ol N_{F\geq2}(r,1;F\mid G\neq 1)+\ol N(r,1;F\mid\geq 2,G\mid=1)\nonumber\\&&+\ol N(r,1;F\mid\geq2,G\mid\geq2)+S(r,G)\nonumber\\&\leq&\ol N_{0}(r,0;G^{\prime})+\ol N(r,1;G\mid\geq3)+\ol N(r,1;G\mid\geq2)+S(r,F)+S(r,G)\nonumber\\&\leq& N(r,0;G^{\prime}\mid G\neq 0)\leq \ol N(r,0;G)+S(r,G).\eea Hence using (\ref{e4.1}), (\ref{e4.2}), Lemmas \ref{lem3.3}, \ref{lem3.8} and \ref{lem3.17}, we get from second fundamental theorem that
 \bea\label{e4.3} && (n+m)T(r,f)\leq T(r,F_1)-N(r,0;L_c(f))+S(r,f)\nonumber\\&\leq& T(r,F)+N_{k+2}(r,0;F_1)-N_2(r,0;F)-N(r,0;L_c(f))+S(r,f) \nonumber\\&\leq& \ol N(r,0;F)+\ol N(r,1;F)+\ol N(r,\infty;F)+N_{k+2}(r,0;F_1)-N_2(r,0;F)\nonumber\\&&-N(r,0;L_c(f))-\ol N_0(r,0;F^{\prime})+S(r,f)\nonumber\\&\leq& \ol N(r,0;F)+\ol N(r,0;F\mid\geq2)+\ol N(r,0;G\mid\geq2)+\ol N(r,1;F\mid\geq2)+\ol N_L(r,1;F)\nonumber\\&&+\ol N_L(r,1;G)+\ol N_0(r,0;G^{\prime})+N_{k+2}(r,0;F_1)-N_2(r,0;F)-N(r,0;L_c(f))\nonumber\\&\leq& N_{k+2}(r,0;F_1)+ N_2(r,0;G)-N(r,0;L_c(f))+S(r,f)+S(r,g)\nonumber\\&\leq& N_{k+2}(r,0;f(z)^nP(f(z))L_c(f))+N_{k+2}(r,0;g(z)^nP(g(z))L_c(g))\nonumber\\&&-N(r,0;L_c(f))+S(r,f)+S(r,g)\nonumber\\&\leq& N_{k+2}(r,0;f^n(z))+N_{k+2}(r,0;P(f))+N_{k+2}(r.0;g^n(z))+N_{k+2}(r,0;P(g))\nonumber\\&&+N(r,0;L_c(g))+S(r,f)+S(r,g)\nonumber\\&\leq& (k+2)\ol N(r,0;f)+N(r,0;P(f))+(k+2)\ol N(r,0;g)+N(r,0;P(g))\nonumber\\&&+N(r,0;L_c(f))+S(r,f)+S(r,g)\nonumber\\&\leq& (k+m+1)(T(r,f)+T(r,g))+T(r,L_c(g))+S(r,f)+S(r,g).\nonumber\eea i.e.,
 \bea &&(n+m)T(r,f)
  \nonumber\\&\leq& (k+m+1)(T(r,f)+T(r,g))+m(r,L_c(g))+S(r,f)+S(r,g)\nonumber\\&\leq& (k+m+1)(T(r,f)+T(r,g))+m\left(r,\frac{L_c(g)}{g}\right)+m(r,g)+S(r,f)+S(r,g)\nonumber\\&\leq& (k+m+1)(T(r,f)+T(r,g))+T(r,g)+S(r,f)+S(r,g).\eea In a similarly, we get \bea\label{e4.4}(n+m)T(r,g)\leq (k+m+1)(T(r,f)+T(r,g))+T(r,f)+S(r,f)+S(r,g). \eea Combining (\ref{e4.3}) and (\ref{e4.4}), we get \beas (n-2k-m-5)(T(r,f)+T(r,g))\leq S(r,f)+S(r,g),\eeas which is a contradicts with  $n\geq 2k+m+6$.\par 
\textbf{ Case 2:} Suppose $H\equiv 0$. Then by integration we get \bea \label{e4.5}F=\frac{AG+B}{CG+D}, \eea where $A,\; B,\; C,\; D$ are complex constant satisfying $AD-BC\neq0$.\par 
\textbf{ Subcase 2.1:} Suppose $AC \neq 0$. Then $F-\frac{A}{C}=\frac{-(AD-BC)}{C(CG+D)}\neq 0.$ So $F$ omits the value $\frac{A}{C}.$\par Therefore, by Lemma \ref{lem3.8} and the Second Fundamental Theorem of Nevalinna, we get 
\beas && (n+m)T(r,f)\leq T(r,f(z)^nP(f(z))L_c(f))-N(r,0;L_c(f))+S(r,f)\\&\leq& T(r,F_1)-N(r,0;L_c(f))+S(r,f)\leq T(r,F)+N_{k+1}(r,0;F_1)-\ol N(r,0;F)\\&&-N(r,0;L_c(f))+S(r,f)\leq \ol N(r,0;F)+\ol N(r,\infty;F)+\ol N\left(r,\frac{A}{C};F\right)\\&&-N(r,0;L_c(f))+N_{k+1}(r,0;F_1)-N(r,0;F)+S(r,f)\\&\leq& N_{k+1}(r,0;f^nP(f)L_c(f))-N(r,0;L_c(f))+S(r,f)\leq (k+1)\ol N(r,0;f)\\&&+N(r,0;P(f))+S(r,f)\leq (k+m+1)T(r,f)+S(r,f),\eeas which is a contradicts with $n\geq 2k+m+6$.\par 
\textbf{ Subcase 2.2:} Suppose $AC=0$. Since $AD-BC\neq 0$, $A$ and $C$ both can not be simultaneously zero.\par 
\textbf{ Subcase 2.2.1:} Let $A\neq 0$ and $C=0$. Then (\ref{e4.5}) becomes $F=A_1G+B_1$, where $A_1=A/D$ and $B_1=B/D$. If $f$ has no 1-point, then by Lemma \ref{lem3.8} and the second fundamental theorem of Nevallina, we get 
\beas && (n+m)T(r,f)\leq T(r,f(z)^nP(f(z))L_c(f))-N(r,0;L_c(f))+S(r,f)\\&\leq& T(r,F_1)-N(r,0;L_c(f))+S(r,f)\leq T(r,F)+N_{k+1}(r,0;F_1)-\ol N(r,0;F)\\&&-N(r,0;L_c(f))+S(r,f)\leq \ol N(r,0;F)+\ol N(r,\infty;F)+\ol N\left(r,1;F\right)\\&&+N_{k+1}(r,0;F_1)-\ol N(r,0;F)-N(r,0;L_c(f))+S(r,f)\\&\leq& \ol N_{k+1}(r,0;f^nP(f)L_c(f))-N(r,0;L_c(f))\leq (k+m+1)T(r,f)+S(r,f),\eeas which is a contradiction since $n\geq 2k+m+6$.	Let $f$ has some 1-point. Then $A_1+B_1=1$. Therefore, $F=A_1G+1-A_1$. If $A_1\neq 1$, then using Lemmas \ref{lem3.8}, \ref{lem3.4}, \ref{lem3.17} and the second fundamental theorem, we get \beas && (n+m)T(r,g)\leq T(r,g^nP(g)L_c(g))-N(r,0;L_c(g))+S(r,g)\\&\leq& T(r,G_1)-N(r,0;L_c(g))+S(r,g)\\&\leq& T(r,G)+N_{k+1}(r,0;G_1)-\ol N(r,0;G)-N(r,0;L_c(g))+S(r,g)\\&\leq& \ol N(r,0;G)+\ol N(r,\infty;G)+\ol N \left(r,\frac{1-A_1}{A_1};G\right)+N_{k+1}(r,0;G_1)-\ol N(r,0;G)\\&&-N(r,0;L_c(g))+S(r,g)\\&\leq& N_{k+1}(r,0;G_1)+ \ol N(r,0;F)-N(r,0;L_c(g))+S(r,g)\\&\leq&  N_{k+1}(r,0;F_1)+N_{k+1}(r,0;G_1)-N(r,0;L_c(g))+S(r,g)\\&\leq&(k+m+1)T(r,f)+T(r,L_c(f))+(k+m+1)T(r,g)+S(r,f)+S(r,g)\\&\leq& (2k+2m+3)T(r,g)+S(r,g).\eeas i.e., 
\beas (n-2k-m-3)T(r,g)\leq S(r,g),\eeas
which is a contradiction since $n\geq 2k+m+6$. Hence $A_1=1$, and therefore we have $ F\equiv G.$ i.e.,\beas (f(z)^nP(f(z))L_c(f))^{(k)}\equiv (g(z)^nP(g(z))L_c(g))^{(k)}.\eeas Integrating $k$ times we get \bea \label{e4.6a}f^nP(f)L_c(f)\equiv g^nP(g)L_c(g)+p(z),\eea where $p(z)$ is a polynomial of degree at most $k-1$. Suppose $p(z) \not\equiv 0$. Then from (\ref{e4.6a}), we have
\bea \label{e4.7a}\frac{f^nP(f)L_c(f)}{p(z)}\equiv \frac{g^nP(g)L_c(g)}{p(z)}+1.\eea Now, using Lemmas \ref{lem3.3}, \ref{lem3.7} and the second fundamental theorem, we get \beas &&(n+m)T(r,f)\leq T(r,f^nP(f)L_c(f))-N(r,0;L_c(f))+S(r,f)\\&\leq& T(r,f^nP(f)L_c(f)/p(z))-N(r,0;L_c(f))+S(r,f)\\&\leq& \ol N\left(r,0;\frac{f^nP(f)L_c(f)}{p}\right)+\ol N\left(r,\infty;\frac{f^nP(f)L_c(f)}{p}\right)+\ol N\left(r,1;\frac{f^nP(f)L_c(f)}{p}\right)\\&&-N(r,0;L_c(f))+S(r,f)\\&\leq& \ol N(r,0;f)+\ol N(r,0;P(f))+\ol N\left(r,0;\frac{g^nP(g)L_c(g)}{p}\right)+S(r,f)\\&\leq& \ol N(r,0;f)+\ol N(r,0;P(f))+\ol N(r,0;g)+\ol N(r,0;P(g))+\ol N(r,0;L_c(g))+S(r,f)\\&\leq& (m+1)T(r,f)+(m+1)T(r,g)+T(r,L_c(g))+S(r,f)+S(r,g)\\&\leq&(m+1)T(r,f)+(m+1)T(r,g)+m\left(r,\frac{L_c(g)}{g}\right)+m(r,g)+S(r,f)+S(r,g)\\&\leq& (m+1)T(r,f)+(m+2)T(r,g)+S(r,f)+S(r,g).\eeas
Similarly, we get \beas (n+m)T(r,g)\leq(m+1)T(r,g)+(m+2)T(r,f)+S(r,f)+S(r,g).\eeas Combining the above two equations, we get 
\beas (n-m-3)(T(r,f)+T(r,g))\leq S(r,f)+S(r,g),\eeas which contradicts to the fact that $n\geq 2k+m+6$. Hence $p(z)\equiv 0$ and so from (\ref{e4.6a}), we obtain \beas f^nP(f)L_c(f)\equiv g^nP(g)L_c(g).\eeas
i.e. \bea\label{e4.6} f^n(a_mf^m+a_{m-1}f^{m-1}+\ldots+a_1f+a_0)(f(z+c)+c_0f(z))\nonumber\\\equiv g^n(a_mg^m+a_{m-1}g^{m-1}+\ldots+a_1g+a_0)(g(z+c)+c_0g(z)).\eea Let $h=f/g$. Then the above equation can be written as \beas && [a_m(h^{n+m}h(z+c)-1)g^m+a_{m-1}(h^{n+m-1}h(z+c)-1)g^{m-1}+\ldots\\&&+a_0(h^{n}h(z+c)-1)]g(z+c)\\&&\equiv -c_0[a_m(h^{n+m+1}-1)g^m+a_{m-1}(h^{n+m}-1)g^{m-1}+\ldots+a_0(h^{n+1}-1)]g(z).\eeas If $h$ is constant, then the above equation can be written as \beas [a_m(h^{n+m+1}-1)g^m+a_{m-1}(h^{n+m}-1)g^{m-1}+\ldots+a_0(h^{n+1}-1)]L_c(g)\equiv0.\eeas Since $L_c(g)\not\equiv0$, we must have \beas a_m(h^{n+m+1}-1)g^m+a_{m-1}(h^{n+m}-1)g^{m-1}+\ldots+a_0(h^{n+1}-1)=0.\eeas Then by a similar argument as in the Case 2 in the proof of Theorem 11 \cite{Xu & Liu & Cao & 2015}, we obtain $f=tg$, where $t$ a constant such that $t^d=1$, $d=\text{gcd}(\lambda_0, \lambda_1,\ldots,\lambda_m)$, where $\lambda_j$'s are defined by \beas \lambda_j=\begin{cases} n+1+j, \;\;\;\text{if}\; a_j\neq0\\ n+1+m, \; \text{if}\; a_j=0, \end{cases}j=0,1,\ldots,m.\eeas\par  If $h$ is not constant, then it follows that $f,\;g$ satisfy the algebraic equation $R(w_1,w_2)=0$, where \beas R(w_1,w_2)=w_1^{n}P(w_1)L_c(w_1)-w_2^{n}P(w_2)L_c(w_2).\eeas
\textbf{ Subcase 2.2.2:} Let $A=0$ and $C\neq0$. Then (\ref{e4.5}) becomes \bea\label{e4.7} F=\frac{1}{A_2G+B_2},\eea where $A_2=C/B$ and $B_2=D/B$. If $F$ has no $1$-point, then by a similar argument as done in \textbf{Subcase 2.2.1}, we can get a contradiction. Let $F$ has some $1$-point. Then $A_2+B_2=1$. If $A_2\neq 1$, then (\ref{e4.7}) can be written as \bea \label{e4.8}F=\frac{1}{A_2G+1-A_2}.\eea Since $F$ is entire and $A_2\neq0$, $G$ omits the value $(1-A_2)/A_2$. Therefore, by Lemma \ref{lem3.8} and the second fundamental theorem, we get \beas && (n+m)T(r,g)\leq T(r,g^nP(g)L_c(g))-N(r,0;L_c(g))+S(r,g)\\&\leq&T(r,G_1)-N(r,0;L_c(g))+S(r,g)\\&\leq& T(r,G)+N_{k+1}(r,0;G_1)-\ol N(r,0;G)-N(r,0;L_c(g))+S(r,g)\\&\leq& \ol N(r,0;G)+\ol N(r,\infty;G)+\ol N \left(r,\frac{1-A_2}{A_2};G\right)+N_{k+1}(r,0;G_1)-\ol N(r,0;G)\\&&-N(r,0;\Delta_cg)+S(r,g).\eeas
i.e., \beas && (n+m)T(r,g)\leq N_{k+1}(r,0;g^nP(g)L_c(g))-N(r,0;L_c(g))+S(r,g)\\&\leq&N_{k+1}(r,0;g^n)+N_{k+1}(r,0;P(g))+S(r,g)\\&\leq&(k+1)\ol N(r,0;g)+N(r,0;P(g))+S(r,g)\leq (k+m+1)T(r,g)+S(r,g),\eeas which is a contradiction since $n\geq2k+m+6$. Hence $A_2=1$. So, from (\ref{e4.8}), we get $FG\equiv 1$. i.e., \bea\label{e4.9} (f(z)^nP(f(z))L_c(f))^{(k)}(g(z)^nP(g(z))L_c(g))^{(k)}\equiv\alpha^2(z).\eea Let $u_1, u_2,\ldots,u_t$, $1\leq t\leq m$ be the distinct zeros of $P(z)$. Since $m\geq k+1$, $a_0\neq0$ and $f$ is entire, it is easily seen from (\ref{e4.9}) that $f$ has atleast two finite Picard exceptional values, which is not possible. Hence the proof is complete. \end{proof}
\begin{proof}[Proof of Theorem \ref{t2}]
Let $F$ and $G$ be defined as in Theorem \ref{t1}.
Then $F$ and $G$ are two transcendental meromorphic functions that share $(1,2)^*$ except the zeros and poles of $\alpha(z)$. We consider the following two cases.\par 
\textbf{Case 1:} Suppose $H\not\equiv 0$. Since $F$ and $G$ share $(1,2)^*$, it follows that $F$ and $G$ share $(1,1)^*$. Also we note that $\ol N(r,1;F\mid=1, G\mid=0)=S(r,F)+S(r,G)$. Keeping in view of Lemmas \ref{lem3.14} and \ref{lem3.15}, we see that \bea\label{e4.10} &&\ol N(r,1;F)=N(r,1;F\mid\leq 1)+\ol N(r,1;F\mid\geq2)\nonumber\\&\leq&\ol N(r,1;F\mid=1, G\mid=0)+\ol N^E(r,1;F,G\mid\leq 1)+\ol N(r,1;F\mid\geq2)\nonumber\\&\leq& N(r,\infty;H)+\ol N(r,1;F\mid\geq2)+S(r,F)+S(r,G)\nonumber\\&\leq& \ol N(r,0;F\mid\geq2)+\ol N(r,0;G\mid\geq2)+\ol N_*(r, 1; F,G)+\ol N(r,1;F\mid\geq2)\nonumber\\&&\ol N_0(r,0;F^{\prime})+\ol N_0(r,0;G^{\prime})+S(r,F)+S(r,G).\eea 
Since $F$, $G$ share $(1,2)^*$, we must have $\ol N_{F\geq2}(r,1;F\mid G\neq1)=S(r,F)+S(r,G)$, $\ol N(r,1;F\mid=2,G\mid=1)=S(r,F)+S(r,G)$. Therefore, using Lemma \ref{lem3.10}, we get
\bea \label{e4.11}&&\ol N(r,1;F\mid\geq2)\nonumber\\&\leq& \ol N_{F\geq2}(r,1;F\mid G\neq1)+\ol N(r,1;F\mid\geq2, G\mid=1)+\ol N(r,1;F\mid\geq2,G\mid\geq2)\nonumber\\&\leq&\ol N_{F\geq2}(r,1;F\mid G\neq1)+\ol N(r,1;F\mid=2,G\mid=1)+\ol N_{F\geq3}(r,1;G\mid=1)\nonumber\\&&+\ol N(r,1;G\mid\geq2)+S(r,F)+S(r,G)\nonumber\\&\leq& \ol N(r,0;F)+\ol N(r,1;G\mid\geq2)+S(r,F)+S(r,G).\eea
Again using (\ref{e4.11}) and Lemma \ref{lem3.3b}, we get \bea\label{e4.12} && \ol N_{0}(r,0;G^{\prime})+\ol N(r,1;F\mid\geq2)+\ol N_*(r,1;F,G)\nonumber\\&\leq&\ol N_{0}(r,0;G^{\prime})+\ol N(r,1;G\mid\geq2)+\ol N(r,1;G\mid\geq3)+\ol N(r,0;F)+S(r,G)\nonumber\\&\leq& \ol N_0(r,0;G^{\prime})+N(r,1;G)-\ol N(r,1;G)+\ol N(r,0;F)+S(r,F)+S(r,G)\nonumber\\&\leq& N(r,0;G^{\prime}\mid G\neq 0)+\ol N(r,0;F)+S(r,F)+S(r,G)\nonumber\\&\leq& \ol N(r,0;F)+\ol N(r,0;G)+S(r,F)+S(r,G).\eea
Hence using (\ref{e4.10}), (\ref{e4.12}), Lemmas \ref{lem3.3} and \ref{lem3.8}, Second fundamental theorem of Nevalinna, we get \bea  && (n+m)T(r,f)\leq T(r,f^nP(f)L_c(f))-N(r,0;L_c(f))+S(r,f)\nonumber\\&\leq& T(r,F)+N_{k+2}(r,0;F_1)-
N_2(r,0;F)-N(r,0;L_c(f))+S(r,f).\nonumber\eea
i.e., \bea \label{e4.13} && (n+m)T(r,f)\nonumber\leq \ol N(r,0;F)+\ol N(r,\infty;F)+\ol N(r,1;F)-\ol N(r,0;F^{\prime})\nonumber\\&&+N_{k+2}(r,0;F_1)\nonumber-N_2(r,0;F)-N(r,0;L_c(f))+S(r,f)
\nonumber\\&\leq& N_2(r,0;F)+N_2(r,0;G)+\ol N(r,0;F)+N_{k+2}(r,0;F_1)-N_2(r,0;F)\nonumber\\&&-N(r,0;L_c(f))+S(r,f)+S(r,g)\nonumber\\&\leq& N_{k+2}(r,0;F_1)+N_{k+2}(r,0;G_1)+N_{k+1}(r,0;F_1)-N(r,0;L_c(f))\nonumber\\&&+S(r,f)+S(r,g)\leq (k+2)(\ol N(r,0;f)+\ol N(r,0;g))+N(r,0;P(f))\nonumber\\&&+N(r,0;P(g))+N(r,0;P(f))+N(r,0;L_c(g))+(k+1)\ol N(r,0;f)\nonumber\\&&+N(r,0;L_c(f))+S(r,f)+S(r,g)\leq(k+m+2)(T(r,f)+T(r,g))\nonumber\\&&+(k+m+1)T(r,f)+T(r,L_c(f))+T(r,L_c(g))+S(r,f)+S(r,g)\nonumber\\&\leq& (k+m+2)(T(r,f)+T(r,g))+(k+m+1)T(r,f)+m\left(r,\frac{L_c(f)}{f}\right)\nonumber\\&&+m(r,f)+m\left(r,\frac{L_c(g)}{g}\right)+m(r,g)+S(r,f)+S(r,g)\nonumber\\&\leq&(k+m+3)(T(r,f)+T(r,g))+(k+m+1)T(r,f)\nonumber\\&&+S(r,f)+S(r,g).\eea In a similar manner, we obtain \bea\label{e4.14} (n+m)T(r,g)&\leq& (k+m+3)(T(r,f)+T(r,g))+(k+m+1)T(r,g)\nonumber\\&&+S(r,f)+S(r,g). \eea Combining (\ref{e4.13}) and (\ref{e4.14}), we get \beas (n-3k-2m+7)(T(r,f)+T(r,g))\leq S(r,f)+S(r,g),\eeas which is a contradicts to the fact that $n\geq 3k+2m+8$.\par \textbf{Case 2:} Let $H\equiv 0$. This case can be carried out similarly as done in  case 2 of the proof of Theorem \ref{t1}. So, we omit the details. This proves Theorem \ref{t2}.
\end{proof}
\begin{proof}[Proof of Theorem \ref{t3}]
Let $F$ and $G$ be defined as in Theorem \ref{t1}. Then $F$ and $G$ are transcendental meromorphic functions such that $E_{2)}(1,F)=E_{2)}(1,G)$ except the zeros and poles of $\alpha(z)$. Let us discuss the following two cases.\par 
\textbf{Case 1:} Let $H\not\equiv 0$. Since $E_{2)}(1,F)=E_{2)}(1,G)$, it follows that $E_{1)}(1,F)=E_{1)}(1,G)$. Keeping in view of Lemmas \ref{lem3.11}, \ref{lem3.13} and \ref{lem3.16}, we see that \bea && \ol N(r,1;F)=N(r,1;F\mid\leq 1)+\ol N(r,1;F\mid\geq2)\nonumber\\&\leq& N(r,H)+\ol N(r,1;F\mid=2)+\ol N_{F\geq3}(r,1;F\mid G\neq1)+\ol N(r,1;F\mid\geq3,G\mid\geq3)\nonumber\\&\leq& N(r,\infty;H)+\ol N(r,1;G\mid=2)+\ol N(r,1;G\mid\geq3)+\ol N_{F\geq3}(r,1;F\mid G\neq1)\nonumber\\&&+S(r,F)+S(r,G)\nonumber\\&\leq& N(r,\infty;H)+\ol N(r,1;G\mid\geq2)+\ol N_{F\geq3}(r,1;F\mid G\neq1)+S(r,F)+S(r,G)\nonumber\\&\leq&\ol N(r,0;F\mid\geq2)+\ol N(r,0;G\mid\geq2)+\ol N_L(r, 1; F)+\ol N_L(r,1;G)+\ol N(r,1;G\mid\geq2)\nonumber\\&&+2\ol N_{F\geq3}(r,1;F\mid G\neq1)+\ol N_{G\geq3}(r,1;G\mid F\neq1)+\ol N_0(r,0;F^{\prime})+\ol N_0(r,0;G^{\prime})\nonumber\\&&+S(r,F)+S(r,G).\nonumber\eea i.e., \bea\label{e4.15} && \ol N(r,1;F)
\nonumber\\&\leq&\ol N(r,0;F\mid\geq2)+\ol N(r,0;G\mid\geq2)+\ol N_L(r, 1; F)+\ol N_L(r,1;G)+\ol N(r,1;G\mid\geq2)\nonumber\\&&+2\ol N_{F\geq3}(r,1;F\mid G\neq1)+\ol N_{G\geq3}(r,1;G\mid F\neq1)+\ol N_0(r,0;F^{\prime})+\ol N_0(r,0;G^{\prime})\nonumber\\&&+S(r,F)+S(r,G)\nonumber\\&\leq& \ol N(r,0;F\mid\geq2)+\ol N(r,0;G\mid\geq2)+\ol N_L(r, 1; F)+\ol N_L(r,1;G)+\ol N(r,1;G\mid\geq2)\nonumber\\&+&\ol N(r,0;F)+\frac{1}{2}\ol N(r,0;G)+\ol N_0(r,0;F^{\prime})+\ol N_0(r,0;G^{\prime})+S(r,F)+S(r,G).\eea 
   Now using Lemma \ref{lem3.3b}, we get
	\bea\label{e4.16} && \ol N_{0}(r,0;G^{\prime})+\ol N(r,1;G\mid\geq2)+\ol N_L(r,1;F)+\ol N_L(r,1;G)\nonumber\\&\leq&\ol N_{0}(r,0;G^{\prime})+\ol N(r,1;G\mid\geq2)+\ol N(r,1;G\mid\geq3)+S(r,G)\nonumber\\&\leq& \ol N_0(r,0;G^{\prime})+N(r,1;G)-\ol N(r,1;G)+S(r,F)+S(r,G)\nonumber\\&\leq& N(r,0;G^{\prime}\mid G\neq 0)\leq \ol N(r,0;G)+S(r,F)+S(r,G).\eea
	Therefore, using (\ref{e4.15}), (\ref{e4.16}), Lemmas \ref{lem3.3} and \ref{lem3.8}, we get from Second Fundamental Theorem that
	\bea && \label{e4.19}(n+m)T(r,f)\leq T(r,f^nP(f)L_c(f))-N(r,0;L_c(f))+S(r,f)\nonumber\\&\leq& T(r,F)+N_{k+2}(r,0;F_1)-
	N_2(r,0;F)-N(r,0;L_c(f))+S(r,f)\nonumber\\&\leq& \ol N(r,0;F)+\ol N(r,\infty;F)+\ol N(r,1;F)-\ol N(r,0;F^{\prime})+N_{k+2}(r,0;F_1)\nonumber\\&&-
	N_2(r,0;F)-N(r,0;L_c(f))+S(r,f)\nonumber\\&\leq& N_2(r,0;F)+N_2(r,0;G)+\ol N(r,0;F)+\frac{1}{2}\ol N(r,0;G)+N_{k+2}(r,0;F_1)\nonumber\\&&-N_2(r,0;F)-N(r,0;L_c(f))+S(r,f)+S(r,g)\nonumber\\&\leq&N_{k+2}(r,0;F_1)+N_{k+2}(r,0;G_1)+N_{k+1}(r,0;F_1)+\frac{1}{2}N_{k+1}(r,0;G_1)\nonumber\\&&-N(r,0;L_c(f))+S(r,f)+S(r,g)\nonumber\\&\leq& (k+2)(\ol N(r,0;f)+\ol N(r,0;g))+2N(r,0;P(f))+N(r,0;P(g))\nonumber\\&&+N(r,0;L_c(g))+(k+1)\ol N(r,0;f)+N(r,0;L_c(f))+\frac{1}{2}(k+1)\ol N(r,0;g)\nonumber\\&&+\frac{1}{2}N(r,0;P(g))+\frac{1}{2}N(r,0;L_c(g))+S(r,f)+S(r,g)\nonumber\\&\leq& (k+m+2)(T(r,f)+T(r,g))+(k+m+1)T(r,f)+T(r,L_c(f))\nonumber\\&&+T(r,L_c(g))+\frac{1}{2}(k+m+1)T(r,g)+\frac{1}{2}T(r,L_c(g))+S(r,f)+S(r,g)\nonumber\\&\leq& (k+m+2)(T(r,f)+T(r,g))+(k+m+1)T(r,f)+m\left(r,\frac{L_c(f)}{f}\right)\nonumber\\&&+m(r,f)+m\left(r,\frac{L_c(g)}{g}\right)+m(r,g)+\frac{1}{2}(k+m+1)T(r,g)+\frac{1}{2}m\left(r,\frac{L_c(g)}{g}\right)\nonumber\\&&+\frac{1}{2}m(r,g)+S(r,f)+S(r,g)\nonumber\\&\leq&(k+m+3)(T(r,f)+T(r,g))+(k+m+1)T(r,f)+\frac{1}{2}(k+m+2)T(r,g)\nonumber\\&&+S(r,f)+S(r,g).\eea Similarly, we obtain \bea\label{e4.20} (n+m)T(r,g)&\leq&(k+m+3)(T(r,f)+T(r,g))+(k+m+1)T(r,g)\nonumber\\&&+\frac{1}{2}(k+m+2)T(r,f)+S(r,f)+S(r,g).\eea Combining (\ref{e4.19}) and (\ref{e4.20}), we get
	\beas \left(n-\frac{7k}{2}-\frac{5m}{2}-8\right)(T(r,f)+T(r,g))\leq S(r,f)+S(r,g),\eeas which is not possible since $n\geq \displaystyle\frac{7k}{2}+\displaystyle\frac{5m}{2}+9$.\par 
	\textbf{Case 2:} Let $H\equiv 0$. This case can be carried out similarly as done in  case 2 of the proof of Theorem \ref{t1}. So, we omit the details. This proves Theorem \ref{t3}.
	\end{proof}

\end{document}